\documentclass[reqno, 11pt]{amsart}
\usepackage{url}

\usepackage{mathrsfs}
\usepackage{lipsum}

\usepackage[colorlinks, linkcolor=blue, citecolor=blue, urlcolor=blue]{hyperref}

\newtheorem{theorem}{Theorem}[section]
\newtheorem{lemma}[theorem]{Lemma}
\newtheorem{proposition}[theorem]{Proposition}
\newtheorem{corollary}[theorem]{Corollary}

\makeatletter
\def\@tocline#1#2#3#4#5#6#7{\relax
  \ifnum #1>\c@tocdepth 
  \else
    \par \addpenalty\@secpenalty\addvspace{#2}%
    \begingroup \hyphenpenalty\@M
    \@ifempty{#4}{%
      \@tempdima\csname r@tocindent\number#1\endcsname\relax
    }{%
      \@tempdima#4\relax
    }%
    \parindent\z@ \leftskip#3\relax \advance\leftskip\@tempdima\relax
    \rightskip\@pnumwidth plus4em \parfillskip-\@pnumwidth
    #5\leavevmode\hskip-\@tempdima
      \ifcase #1
       \or\or \hskip 1em \or \hskip 2em \else \hskip 3em \fi%
      #6\nobreak\relax
    \hfill\hbox to\@pnumwidth{\@tocpagenum{#7}}\par
    \nobreak
    \endgroup
  \fi}
\makeatother

\theoremstyle{definition}

\newtheorem{definition}[theorem]{Definition}

\newtheorem{remark}[theorem]{Remark}

\newtheorem{prob}[theorem]{Problem}
\newtheorem{notation}[theorem]{Notation}
\newtheorem{convention}[theorem]{Convention}

\numberwithin{equation}{section}

\usepackage{bbm}
\usepackage{euscript}
\usepackage{amsmath}
\usepackage{amsthm}

\usepackage{amsxtra}
\usepackage{amssymb}
\usepackage{pifont}
\usepackage{amsbsy}

\usepackage{graphicx}
\usepackage{epstopdf}
\usepackage{multicol}
\usepackage{enumerate}
\usepackage{enumitem}
\usepackage{bbm}
\usepackage{tikz-cd}

\newlist{steps}{enumerate}{1}
\setlist[steps, 1]{label = \textit{Step \arabic*:}}

\newskip\aline \newskip\halfaline
\newcommand\blfootnote[1]{%
  \begingroup
  \renewcommand\thefootnote{}\footnote{#1}%
  \addtocounter{footnote}{-1}%
  \endgroup}

\newcommand{\bR}{\mathbb{R}}
\newcommand{\bZ}{\mathbb{Z}}

\newcommand{\bC}{\mathbb{C}}

\newcommand{\bk}{\mathbbm{k}}
\newcommand{\p}{\mathbf}
\newcommand{\beq}{\begin{equation}}
\newcommand{\eeq}{\end{equation}}
\newcommand{\Om}{\Omega}

\newcommand{\mT}{\mathcal{T}}

\newcommand{\mM}{\mathbb{M}}
\newcommand{\mS}{\mathcal{S}}
\newcommand{\emp}{\emptyset}
\newcommand{\inv}{^{-1}}
\newcommand{\tit}{\textit}

\newcommand{\rsa}{\rightsquigarrow}
\newcommand{\bcomd}{ \begin{center} \begin{tikzcd}}
\newcommand{\ecomd}{ \end{tikzcd} \end{center}}
\newcommand{\sg}{\Sigma}
\makeatletter
\newcommand{\ssymbol}[1]{^{\@fnsymbol{#1}}}
\makeatother

\newcommand{\pl}{\partial}
\newcommand{\x}{\times}
\newcommand{\bpf}{\begin{proof}}
\newcommand{\epf}{\end{proof}}

\newcommand{\bseq}{\beq \begin{split}}

\newcommand{\cl}{\overline}

\newcommand{\bal}{\begin{equation} \begin{aligned}}
\newcommand{\eal}{\end{aligned} \end{equation}}
\newcommand{\bdf}{\begin{definition}}
\newcommand{\bprop}{\begin{proposition}}
\newcommand{\bcor}{\begin{corollary}}
\newcommand{\bclm}{\begin{claim}}
\newcommand{\bthm}{\begin{theorem}}
\newcommand{\bex}{\begin{example}}
\newcommand{\brmk}{\begin{remark}}
\newcommand{\bprob}{\begin{prob}}
\newcommand{\bnot}{\begin{notation}}
\newcommand{\bconv}{\begin{convention}}
\newcommand{\bobs}{\begin{observation}}
\newcommand{\bit}{\begin{itemize}}
\newcommand{\ben}{\begin{enumerate}}
\newcommand{\bcas}{\begin{cases}}
\newcommand{\blem}{\begin{lemma}}
\newcommand{\bscl}{\begin{scholium}}
\newcommand{\edf}{\end{definition}}
\newcommand{\eprop}{\end{proposition}}
\newcommand{\ecor}{\end{corollary}}
\newcommand{\eclm}{\end{claim}}
\newcommand{\ethm}{\end{theorem}}
\newcommand{\eex}{\end{example}}
\newcommand{\ermk}{\end{remark}}
\newcommand{\eprob}{\end{prob}}
\newcommand{\enot}{\end{notation}}
\newcommand{\econv}{\end{convention}}
\newcommand{\eobs}{\end{observation}}
\newcommand{\eit}{\end{itemize}}
\newcommand{\een}{\end{enumerate}}
\newcommand{\ecas}{\end{cases}}
\newcommand{\elem}{\end{lemma}}
\newcommand{\escl}{\end{scholium}}

\DeclareMathOperator{\rk}{rank}
\DeclareMathOperator{\Hom}{Hom}

\title[\resizebox{4.5in}{!}{Relating Cut and Paste Invariants and TQFTs}]{Relating Cut and Paste Invariants and TQFTs}

\author{Carmen Rovi, Matthew Schoenbauer}

\address{Department of Mathematics, Indiana University Bloomington, Rawles Hall, 831 E 3rd St, Bloomington, IN 47405}

\email{crovi@indiana.edu}\address{Department of Mathematics, University of Notre Dame, Notre Dame, IN 465561-4618.}
\email{mschoenb@nd.edu}
\begin{document}

\maketitle


\blfootnote{This work is supported by NSF REU Grant number $1461061$.}

\section*{Abstract}


 In this paper we shall be concerned with a relation between TQFTs and cut and paste invariants introduced in \cite{Karras}. Cut and paste  invariants, or $SK$ invariants, are functions on the set of smooth manifolds that are invariant under the cutting and pasting operation. Central to the work in this paper are also $SKK$ invariants, whose values on cut and paste equivalent manifolds differ by an error term depending only on the glueing diffeomorphism.
Here we investigate a surprisingly natural group homomorphism between the group of invertible TQFTs and the group of $SKK$ invariants and describe how these groups fit into an exact sequence. We conclude in particular that all positive real-valued $SKK$ invariants can be realized as restrictions of invertible TQFTs.

All manifolds are smooth and oriented throughout unless stated otherwise.

\section{Introduction}

\subsection{Cut and Paste Invariants}

 Our knowledge of cut and paste invariants is 
 fruit of several decades of exciting developments in differential and algebraic topology. To understand these invariants, mathematicians needed a solid grasp of characteristic classes, cobordism theory, and the signature. The motivation to study these invariants came from the study of the index of elliptic operators.
 
 The 1950s and 1960s saw incredible developments in all of these areas. In 1953, Vladimir Rokhlin uncovered much about the signature; he found a connection between Pontryagin classes and the signature and discovered the signature's importance in cobordism theory. Crucial here is also Thom's cobordism theory, which he developed during those same years. It was also at this time when Friedrich Hirzebruch proved his famous theorem relating the signature and the Pontryagin numbers, an immensely important development in the theory of characteristic classes. Work continued throughout the next decade, and Sergei Novikov first presented the additivity property of the signature in 1966. Once this result was published in 1970, our knowledge of the signature and cobordism was firm enough for work to be done on the $SK$ and $SKK$ invariants.
 
However, it was the study of the index of elliptic operators that actually brought attention to the study of these invariants. This topic had also been substantially developed in this time period, in particular by Michael Atiyah and Isadore Singer who proved the Atiyah-Singer Index Theorem in 1963. This was a groundbreaking development in both analysis and topology, and showed a deep connection between the two fields. In an attempt to investigate certain aspects this work, Klaus Jaenich noticed that the index of elliptic operators had the properties of invariants that were not influenced by cut and paste operations. In 1968 and 1969, he wrote two papers \cite{Janich}, \cite{Janich2} studying these invariants. These papers included very interesting results but lacked a systematic approach to these ideas.

This opened the door for four young authors, Ulrich Karras, Matthias Kreck, Walter Neumann, and Erich Ossa, to give these invariants a more thorough treatment. These four authors published a short book in 1973 titled \tit{Cutting and Pasting of Manifolds; SK Groups}, in which they completely classified these invariants \cite{Karras}. They were the first to use the terms ``$SK$ invariants" and ``$SKK$ invariants." The ``$SK$" stands for ``schneiden" and ``kleben," which mean ``cut" and ``paste" in German. The second $K$ in $SKK$ stands for ``kontrollierbar," the German word for ``controllable."


\subsection{Topological Quantum Field Theories}

The idea of Topological Quantum Field Theories originated in the 1980s, an era of rapid development of our understanding of the relationship between geometry and physics. The most credit for the initial development of TQFTs is due to Edward Witten and Michael Atiyah. Both did work in mathematics and theoretical physics; Witten was more the physicist and Atiyah more the mathematician. Witten contributed by laying the mathematical foundation for super-symmetric quantum mechanics. This theory is rather complicated, and Atiyah contributed by looking at a simplified, purely mathematical version of Witten's theories. In \cite{Atiyah}  Atiyah axiomatized TQFTs, which, unlike Witten's theory, did not involve any geometric ideas such as curvature or Riemannian metric.

\subsection{The origin of this project}
The relation between $SKK$ invariants and invertible TQFTs was first investigated by Matthias Kreck, Stephan Stolz and Peter Teichner. The exact sequence presented in this paper was due to them, even if they had not published their work.

\subsection{Acknowledgements} 
The second named author would like to give great thanks to his REU advisor Dr Carmen Rovi for mentoring the research project that has given rise to this paper. He would like also like to thank the REU organization in Bloomington for the very enjoyable experience during the summer of 2017. In addition to this, he would like to thank Prof Frank Connolly for improving his mathematical maturity over the last year. Both authors would like to thank Prof Chris Schommer-Pries for useful conversations about certain ideas needed to prove the existence of the desired split exact sequence.


\section{Cut and Paste}

In this section we will describe an equivalence relation on manifolds called the ``cut and paste" relation. We will conclude by imposing further structure on these equivalence classes to form the $SK$ and $SKK$ groups. 

We first briefly describe the cut and paste operation on oriented manifolds. To perform this operation on a manifold $M$, cut $M$ along a codimension-1 submanifold $\sg$ with trivial normal bundle and paste the resulting manifold back together via an orientation-preserving diffeomorphism $f:\sg \to \sg$.

\begin{definition} \label{two} \label{01}Two closed oriented manifolds $M$ and $N$ are said to be \tit{cut and paste equivalent} or \tit{SK equivalent} if $N$ can be obtained from $M$ by a finite sequence of cut and paste operations. In this case we write $[M]_{SK}=[N]_{SK}$.
\end{definition}

This is clearly an equivalence relation. A pictorial representation of a nontrivial cut and paste operation is shown in Figure \ref{f2}. The figures on the top right and bottom are mapping tori of the map $f: \sg \to \sg$. 

\begin{figure} [h] 
   \centering
   \def\svgwidth{\columnwidth}
   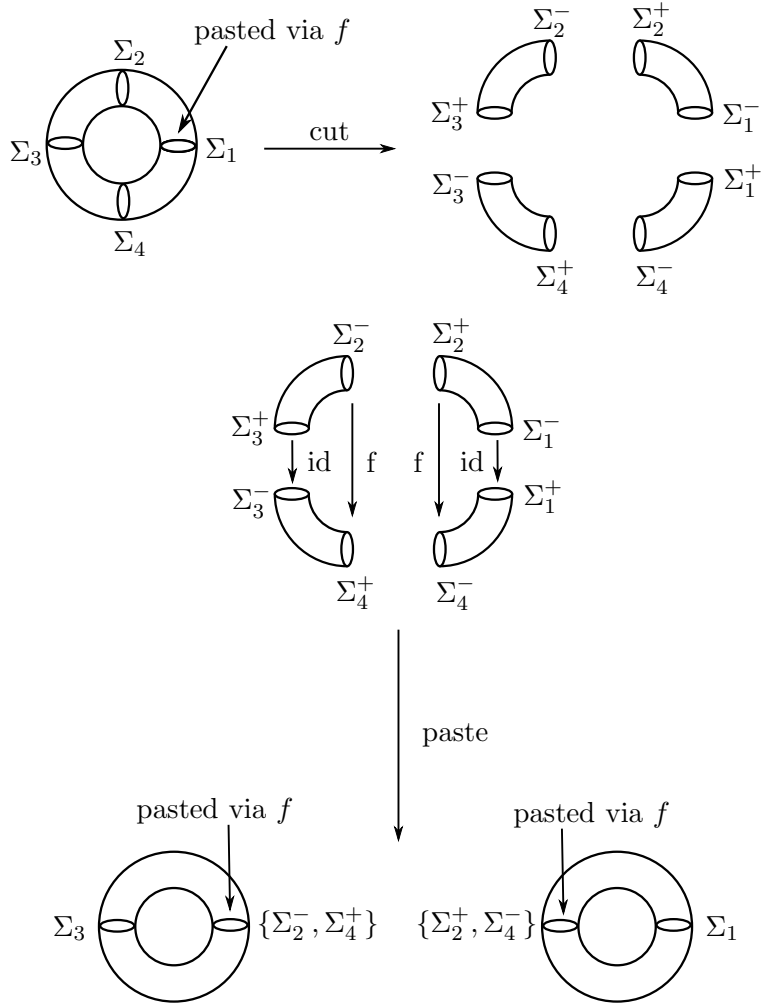
\caption{Cutting and pasting a mapping torus}
\label{f2}
\end{figure}

We now describe the cut and paste invariants.  Let $\mathbb{M}_n \label{001}$ denote the set of all closed oriented $n$-manifolds, and let $M,N \in \mathbb{M}_n$.

\begin{definition} Let $G$ be an abelian group. A function $\Theta:  \mathbb{M}_n \to G$ is said to be an $n$-dimensional {\textit{cut and paste invariant}} or \tit{SK invariant} if the following hold. \begin{itemize} \item $\Theta(M)=\Theta(N)$ whenever $M$ and $N$ are cut and paste equivalent. \item $\Theta(M \coprod N) = \Theta(M) + \Theta(N)$ \end{itemize}
\end{definition}

It follows from the additivity properties of the signature and Euler characteristic that both are $SK$ invariants.

We can also define a weaker class of invariants, which will be especially interesting for later use.

\begin{definition} \label{SKK} Let $G$ be an abelian group.  A function $\xi:  \mathbb{M}_n \to G$ is said to be an $n$-dimensional \tit{SKK invariant} if the following hold. \begin{itemize} \item  Suppose $M$ and $N$ are cut and paste equivalent, i.e. $M= X_1 \cup_{f} X_2$ and $N= X_1 \cup_{g} X_2$ so that $f$ and $g$ are the gluing diffeomorphisms of $M$ and $N$, respectively.  Then $$\xi(M)-\xi(N)=\xi(f,g)$$ where $\xi(f,g) \in G$ depends only on  $f$ and $g$. \item $\xi(M \coprod N) = \xi(M) + \xi(N) $ \end{itemize} \end{definition} 

It is clear from the definitions that every $SK$ invariant is an $SKK$ invariant. For an $SK$ invariant, the ``error" function $\xi(f,g)$ is always zero.
Note that the converse is not true, since there exist $SKK$ invariants which are not $SK$ invariants. One such example is the Kervaire semicharacteristic. 

\begin{definition} Let $M$ be a compact, oriented $n$-dimensional manifold. We define the \tit{Kervaire semicharacteristic} $\chi_{1/2}$ by \beq \chi_{1/2}(M)= \begin{cases} \frac{1}{2} \chi(M) & \mathit{n} \ \mathrm{even} \\[.5em]
\big( \sum \limits _{i=0}  \rk{H_{2i}(M)}\big) \  \mathrm{mod} \ 2 & \mathit{n} \ \mathrm{odd}, \end{cases}  \eeq where $\chi$ is the Euler characteristic.
\end{definition}

 See \cite[p.~44]{Karras} for a proof that $\chi_{1/2}$ is an $SKK$ invariant in dimensions $4n+1$. It is an $SKK$ invariant in all dimensions other than dimensions $4n-1$.

In order to classify $SK$ and $SKK$ invariants, we must define the $SK$ and $SKK$ groups. $ \mathbb{M}_n$ is a  communtative monoid under the disjoint union operation, so we can form its Grothendieck group $\mathcal{G}(\mathbb{M}_n) \label{002}$. We  form quotient groups of this group, which will give us more concise and useful definitions of $SK$ and $SKK$ invariants. 

\begin{definition} Let $R^{SK}_n$ denote the subgroup of $\mathcal{G}(\mathbb{M}_n)$ generated by all elements of the form $[M] \coprod -[N]$ where $[M ]_{SK}=[N]_{SK}$. $\mathcal{G}(\mathbb{M}_n)/R_n^{SK}=SK_n$ is called the {\textit{$n^{th}$ $SK$ group}}. \label{03}
\end{definition}

This definition gives us an alternative description of $SK$ invariants.

\begin{proposition} Let $G$ be an abelian group. A function $\Theta:  \mathbb{M}_n \to G$ is an $n$-dimensional $SK$ invariant if and only if it is an element of the group $ \Hom{(SK_n, G)}$. \end{proposition}

We can now give a full description of $SK$ invariants. In the following theorem, $\sigma$ denotes the signature and $\chi$ denotes the Euler characteristic.

\begin{theorem} \cite[p.~7]{Karras}  \label{phew} 

\begin{itemize}
\item[(a)] For $n$ odd, we have \beq SK_n= 0. \eeq 

\item[(b)]For $n \equiv 0$ mod $4$, we have \beq \left(\frac{\chi - \sigma}{2}\right) \oplus \sigma : SK_n \xrightarrow{\cong} \bZ \oplus \bZ, \eeq so the signature $\sigma$ and the Euler characteristic $\chi$ together form a complete set of invariants. 

\item[(c)] For $n \equiv 2$ mod $4$, we have \beq \frac{\chi}{2} : SK_n \xrightarrow{\cong} \bZ, \eeq so $\chi$ is a complete invariant.
\end{itemize}
\end{theorem}

We would like to construct a quotient group of $\mathcal{G}(\mathbb{M}_n)$ that serves the same purpose as $SK_n$, but for $SKK$ invariants.  How can this be done? The stated condition \beq \xi([M])-\xi([N])=\xi(f,g) \eeq when $[M]_{SK}= [N]_{SK}$ and $f$ and $g$ are the gluing diffeomorphisms of $M=X_1 \cup_f X_2$ and $N= X_1 \cup_g X_2$, respectively, is equivalent to the condition \beq \xi([M])-\xi([N])= \xi([M'])-\xi([N']) \eeq when  $[M']_{SK} =[N']_{SK}$ and  $f$ and $g$ are also the gluing diffeomorphisms of $M' = Y_1 \cup_f Y_2$ and $N' = Y_1 \cup_g Y_2$, respectively. This leads to the following definition. 

\begin{definition}  Let $R_n^{SKK}$ denote the subgroup of $\mathcal{G}(\mathbb{M}_n)$ generated by all elements of the form \beq [M] \coprod -[N] \coprod -[M'] \coprod [N'] , \eeq  where $[M]_{SK} =[N]_{SK}$, $[M']_{SK} =[N']_{SK}$ and the gluing diffeomorphisms are given by the above description. $\mathcal{G}(\mathbb{M}_n)/R_n^{SKK}=SKK_n$ is called the {\textit{$n^{th}$ $SKK$ group}}.  \label{04}
\end{definition}

Just as in the case of $SK_n$, we will let $[M]_{SKK} \label{02}$ denote the $SKK$ equivalence class of a closed manifold $M$. We now get our desired description of $SKK$ invariants.

\begin{proposition}  Let $G$ be an abelian group. An $n$-dimensional $SKK$ invariant is an element of the group $ \Hom{(SKK_n, G)}$. \end{proposition} 

\section{Cobordisms and $SKK$ Groups}

In this section we briefly describe cobordisms, which is needed both for the classification of $SKK$ invariants and the definition of TQFTs.

\begin{definition} Let $ \sg_0$ and $\sg_1$ be closed oriented $(n-1)$-manifolds. An  \tit{{}{oriented n-cobordism}} $M:\sg_0 \rsa \sg_1$ is a manifold $M$ along with orientation-preserving diffeomorphisms  \beq \phi_{\textnormal{in}}:\sg_0 \to \overline{\pl_{\textnormal{in}}M} \ \  \mathrm{and} \ \  \phi_{\textnormal{out}}:\sg_1 \to \pl_{\textnormal{out}}M,\eeq where \beq \pl M = \pl_{\textnormal{in}} M \coprod \pl_{\textnormal{out}} M.\eeq  Here $\pl_{\textnormal{in}}M$ is called the \tit{in-boundary} of $M$ and $\pl_{\textnormal{out}}M$ is called the \tit{out-boundary} of $M$. Similarly, $\phi_{\textnormal{in}}$ and $\phi_{\textnormal{out}}$ are call the \tit{in-boundary diffeomorphism} of $M$ and the \tit{out-boundary diffeomorphism} of $M$, respectively.
\end{definition}

The cobordism relation is an equivalence relation and the set $\Omega_{n-1}$ consisting of all oriented cobordism classes of $(n-1)$-dimensional manifolds forms an abelian group with disjoint union as composition operation.

\medskip

Note that there is a natural gluing operation on two cobordisms $M: \sg_0 \rsa \sg_1$ and $N: \sg_1 \rsa \sg_2$. The gluing diffeomorphism $\pl_{\textnormal{out}} M \to \overline{\pl_{\textnormal{in}} N}$ is the map $\phi'_{\textnormal{in}} \circ \phi_{\textnormal{out}} \inv$, where $\phi'_{\textnormal{in}}$ is the in-boundary diffeomorphism $\sg_1 \to \overline{\pl_{\textnormal{in}}N}$ of $N$. We will denote the resultant cobordism $MN$.

\begin{definition} Two cobordisms $M: \sg_0 \rsa \sg_1$ and $N: \sg_0 \rsa \sg_1$ are said to be \tit{{}{equivalent}} if there exists an orientation-preserving diffeomorphism $\psi$ making the following diagram commute.
\bcomd & M & \\
\sg_0 \arrow{ur}{\phi_{\textnormal{in}}} \arrow{dr}[swap]{\phi'_{\textnormal{in}}} & & \sg_1 \arrow{dl}{\phi'_{\textnormal{out}}} \arrow{ul}[swap]{\phi_{\textnormal{out}}} \\
& N \arrow{uu}[swap]{\psi}& \label{40}
\ecomd
In this case we write $M \sim N$.

\end{definition}




We now relate the cobordism groups  to the $SKK$ groups. 

\begin{theorem} \cite[p.~44]{Karras} The homomorphism $SKK_n \to \Omega_n$ that assigns to each manifold its cobordism class is a surjective $SKK$ invariant. \end{theorem}

\begin{theorem} \cite[p.~44]{Karras} \label{class} The following sequence is exact: \beq 0 \to I_n \to SKK_n \to \Omega_n \to 0, \eeq where \beq I_n = \begin{cases} \bZ  & n \equiv 0 \ \mathrm{mod} \ 2 \\  \bZ_2 & n \equiv 1 \ \mathrm{mod} \ 4 \\ 0 & n \equiv 3 \ \mathrm{mod} \ 4 \end{cases} \eeq In addition, $\chi$ splits the sequence in dimensions divisible by $4$, and $\chi_{1/2}$ splits the sequence in dimensions $n \equiv 1,2 \ \mathrm{mod} \ 4$.
\end{theorem}

Thus the $SKK$ invariants are none other than linear combinations of the Euler characteristic, the Kervaire semicharacteristic, and bordism invariants. 


\section{Topological Quantum Field Theories and the Group of Invertible TQFTs \label{TQFT} }

In this section we give the definition of a TQFT. Furthermore we will describe a multiplication operation on TQFTs, which will allow us to define the group of invertible TQFTs.

\begin{definition} \label{dTQFT}An \tit{{}{n-dimensional oriented TQFT}} is a symmetric monoidal functor $\mT$ from the $n$-cobordism category to the vector space category, assigning to each closed oriented $(n-1)$-manifold $\sg$ a $\bk$-vector space $\mT(\sg)$ and to each oriented $n$-dimensional cobordism $M: \sg_0 \to \sg_1$ a linear map $\mT(M): \mT(\sg_0) \to \mT(\sg_1)$, satisfying the following properties:
\begin{enumerate}
\item Two equivalent cobordisms have the same image. \beq M \sim N \implies \mT(M)=\mT(N)\eeq
\item A glued cobordism goes to a composition of linear maps. \beq \mT(MN)= \mT(M) \circ \mT(N)\eeq
\item A cylinder cobordism gets sent to the identity map. \beq\mT(\sg \times I)= Id_{\mT(\sg)}\eeq
\item Disjoint unions of $(n-1)$-manifolds and cobordisms get sent to a tensor product of vector spaces and linear maps, respectively. \beq \mT(\sg \coprod \sg')= \mT(\sg) \otimes \mT(\sg') \eeq \beq \mT(M \coprod N) = \mT(M) \otimes \mT(N) \eeq
\item The empty manifold $\emptyset$ gets sent to the ground field $\bk$. \beq \mT(\emptyset)=\bk \eeq
\end{enumerate}
\end{definition}


Similarly to the case of cobordisms, we will refer to all oriented TQFTs as ``TQFTs" for brevity.

\medskip

We can now form a \textit{category} of TQFTs, which we will denote by $\p{nTQFT_\bk} \label{05}$, where the objects are TQFTs and the arrows are natural transformations of TQFTs. This is a symmetric monoidal functor category, which has a natural symmetric monoidal structure. 

 \begin{definition} The product of two TQFTs $\mathcal{T}_1$ and $\mathcal{T}_2$ is the TQFT \beq \mT_1 \otimes \mT_2 \eeq that assigns to each $(n-1)$-manifold $\Sigma$ the vector space \beq \mathcal{T}_1(\Sigma) \otimes \mT_2(\Sigma) \eeq and to each $n$-cobordism $M: \Sigma_0 \rsa \Sigma_1$ the linear map  \beq \mT_1{(M)} \otimes \mT_2{(M)} : \mT_1{(\sg_0)} \otimes \mT_2{(\sg_0)} \to \mT_1{(\sg_1)} \otimes \mT_2{(\sg_1)}.\eeq \end{definition} 
 
The trivial TQFT, which sends each $(n-1)$-manifold to $\bk$ and each $n$-cobordism to the identity, is an identity element under the tensor product operation. 

We wish to study the invertible objects in $\p{nTQFT_\bk}$, that is, the TQFTs  $\mT$ with an inverse $\mT'$ such that $\mT \otimes \mT'$ is the trivial TQFT. It is a consequence of the axioms of TQFTs that TQFTs only assign finite-dimensional vector spaces to closed $(n-1)$-manifolds  (See \cite[p.~31]{Kock}). Since tensoring multiplies the dimension of finite-dimensional vector spaces, each vector space that an invertible TQFT $\mT$ assigns to an $(n-1)$-manifold must be 1-dimensional, i.e. isomorphic to $\bk$. 

Since all linear maps $\bk \to \bk$ are simply scalar multiplication, each map that an invertible TQFT $\mT$ assigns to an $n$-manifold can be canonically associated with a scalar. All linear maps assigned by invertible TQFTs must be invertible, that is, multiplication by a nonzero scalar.

It is clear that the set of invertible TQFTs forms a group under the composition operation in $\p{nTQFT}_\bk$. We will denote this group as $\p{nTQFT}_\bk^{\x} \label{06}$.

What we will do next combines many of the ideas that we have presented so far. Our goal is to determine how invertible TQFTs evaluate two cut and paste equivalent closed manifolds, considered as cobordisms $\emp \rsa \emp$. 

Let $\mT$ be an invertible TQFT and let $M$ and $N$ be cut and paste equivalent closed manifolds, with gluing diffeomorphisms $f$ and $g$, respectively, as in Figure \ref{43}. 

\begin{figure} [h]
   \centering
   \def\svgwidth{\columnwidth}
   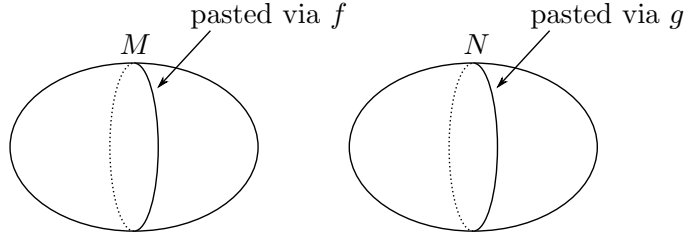
    \caption{Cut and paste equivalent manifolds}
   \label{43}
\end{figure}

 Also, let $\delta$ be the canonical isomorphism $\Hom_\bk{(\bk, \bk)} \to \bk$, where $\Hom_\bk{(\bk, \bk)} $ is the vector space of $\bk$-linear automorphisms of $\bk$. Now using the fact that $\Sigma$ has a collar neighborhood in both $M$ and $N$, we can replace $M$ and $N$ with the equivalent cobordisms \beq M_1C_fM_2 \ \ \mathrm{and} \ \ M_1C_gM_2, \eeq where $C_f$ and $C_g$ denote the mapping cylinders of $f$ and $g$ respectively, as in Figure \ref{44}.
 
 \begin{figure} [h]
   \centering
   \def\svgwidth{\columnwidth}
   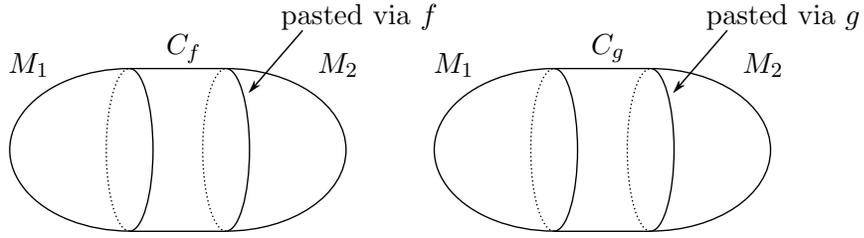
   \label{44}
   \caption{Equivalent cobordisms}
\end{figure}

Now evaluating $\mT$ on both cobordisms and taking a quotient, we see an interesting relation. \beq \frac{\delta(\mT(M_1C_fM_2))}{\delta(\mT(M_1C_gM_2))}= \frac{\delta(\mT(M_1))\cdot \delta(\mT(C_f)) \cdot \delta(\mT(M_2))}{\delta(\mT(M_1))\cdot \delta(\mT(C_g)) \cdot \delta(\mT(M_2))}= \frac{\delta(\mT(C_f))}{\delta(\mT(C_g))} \eeq Note that the above equation is valid because all scalars are required to be nonzero by the invertibility of $\mT$.

\section{Results}

We now relate the ideas of invertible TQFTs and $SKK$ invariants. Let  $\bk^\x$ be the multiplicative group on $\bk-\{0\}$. The relation will be expressed by means of a group homomorpism \beq \Psi_n: \p{nTQFT}^\x_{\bk} \to \Hom{(SKK_n, \bk^\x)}. \eeq  

Let $\mT \in \p{nTQFT}^\x_{\bk}$. If $X$ is a smooth $n$-dimensional manifold, then we set $\Psi_n(\mT)([X])=\delta(\mT([X]))$.
Now let $M$ and $N$ be cut and paste equivalent closed oriented $n$-manifolds, where $f$ and $g$ are the gluing diffeomorphisms of $M$ and $N$, respectively. We have shown that \beq \frac{\Psi_n(\mT)(M)}{\Psi_n(\mT)(N)}= \frac{\delta(\mT(M))}{\delta(\mT(N))}=\frac{\delta(\mT(C_f))}{\delta(\mT(C_g))}=\Psi_n(\mT)(f,g). \eeq The notation is a bit different since we are dealing with a multiplicative group, but this is precisely the first criterion in the description of $SKK$ invariants in Definition \ref{SKK}.

 Thus $\Psi_n(\mT)$ induces a homomorphism $SKK_n \to \bk^\x$. 
 Hence $\Psi_n(\mT)$ is an $SKK$ invariant on $n$-manifolds, and $\Psi_n$ is a homomorphism $\p{nTQFT}^\x_{\bk} \to \Hom{(SKK_n, \bk^\x)}$. 

This homomorphism $\Psi_n$ gives a natural relationship between invertible TQFTs and $SKK$ invariants. Note that this relationship only makes sense if $G$ from Definition \ref{SKK} is the multiplicative group of a field. This causes a problem for the $SKK$ invariants we have given, since the target groups here include $\bZ$, which is not isomorphic to the multiplication group of any field. However, the problem is solved by the exponential map $\exp: x \to e^x$. This exponential map $\exp$ sends $\bZ$ monomorphically to a subgroup of $\bR^\x$. This allows us to think of the our additive integer and rational-valued invariants as elements of $\Hom{(SKK_n,\bR^\x)}$.


We will be interested in finding the kernel and image of $\Psi_n$. The kernel is easy to describe.

\begin{theorem} Let $\mathscr{T}^n_\bk$ be the  subgroup  of $\p{nTQFT}^\x_{\bk}$ that consists of all $\mT \in  \p{nTQFT}^\x_{\bk}$  such that $\mT(M)= id_{\bk}$ for all $n$-cobordisms $M$ with empty boundary. Then $\mathscr{T}^n_\bk= \ker(\Psi_n)$.
\end{theorem}
\begin{proof} Let $\mT$ be an  invertible $n$-TQFT, and suppose $\Psi_n(\mT)$ is the trivial $SKK_n$-invariant, i.e. the invariant that sends all closed $n$-manifolds to $1\in \bk^\x$. Then obviously $\mT \in \mathscr{T}^n_\bk$. It is also clear that $\mT \in \mathscr{T}^n_\bk \implies \mT \in \ker{\Psi_n}$.
\end{proof} The proposition above expresses a degree of ``forgetfulness" of $\Psi_n$. TQFTs assign values to all cobordisms, with or without boundary, and $SKK$ invariants are only defined on closed manifolds. Thus by applying $\Psi_n$ to $\mT$, one loses some information about $\mT$, as we will show in Theorem \ref{nontriv}.

\medskip
We start by proving the following Lemma.

\blem \label{lem1} If $\mT \in \mathscr{T}^n_\bk$, then for any cobordism $M$, $\mT(M)$ only depends on the in-boundary and out-boundary of $M$.
\elem

\bpf Let $M: \sg_0 \rsa \sg_1$  and $N: \sg_0 \rsa \sg_1$ be cobordisms. Consider the cobordism $\cl{M}: \cl{\sg}_0 \rsa \cl{\sg}_1 $. Also let $$B_{\sg_0}: \emp \rsa \sg_0 \coprod \cl{\sg}_0 $$ and $$B_{\sg_1} : \sg_1 \coprod \cl{\sg}_1 \rsa \emp $$ be the cobordisms with cylinder $n$-manifolds and identity boundary diffeomorphisms. We have $$\delta \inv(1) = \mT(B_{\sg_0}(M \coprod \cl{M}) B_{\sg_1})= \mT(B_{\sg_0})\mT(M ) \mT(\cl{M}) \mT(B_{\sg_1}) $$ so $\mT(M)= (\mT(B_{\sg_0}) \mT(\cl{M}) \mT(B_{\sg_1})) \inv$. The same reasoning shows that $\mT(N)= (\mT(B_{\sg_0}) \mT(\cl{M}) \mT(B_{\sg_1})) \inv$.
\epf

\medskip

\begin{theorem} \label{nontriv} $\mathscr{T}^n_\bk$ is trivial if and only if $\bk^\x$ is trivial. \end{theorem}


\bpf First suppose $\bk^\x$ is trivial. Then there is obviously only one possible invertible TQFT, which is the trivial TQFT. In this case $\mathscr{T}^n_\bk$ is clearly trivial. 

Now suppose $\bk^\x$ is nontrivial. We must define $\mT \in  \p{nTQFT}^\x_{\bk}$ so that $\mT$ evaluates closed manifolds trivially and manifolds with boundary nontrivially. We proceed as follows. 

Let $\mT(\sg)=\bk$ for all $(n-1)$-manifolds $\sg$. Then assign to each closed connected $(n-1)$-manifold $\sg$ a scalar $\lambda_{\sg}$. We require $\lambda_\sg \notin \{0,1\}$ if  $\sg \neq \emp$, and $\lambda_\emp=1$. Now for any cobordism \beq M:\coprod \limits _{i=0} ^n \sg_i \rsa \coprod \limits _{j=0} ^m \sg_j \eeq we define \beq \label{12} \mT(M)=\delta \inv\Big(\prod \limits _{i=0} ^n \lambda_{\sg_i} \cdot \prod \limits _{j=0} ^m\lambda_{\sg_j}\inv\Big). \eeq 

Now we check the axioms of Definition \ref{dTQFT} in order. The value $\mT$ assigns to an $n$-cobordism depends only on its in-boundary and out-boundary. Equivalent cobordisms must have the same in-boundary and out-boundary, so (1) holds. 

Now consider axiom (2), and let  \beq N:\coprod\limits _{j=0} ^m \sg_j   \rsa \coprod \limits _{k=0} ^l \sg_k\eeq be another cobordism. We have 
\beq \begin{split} \mT(MN) &= \delta \inv\Big(\prod \limits _{i=0} ^n \lambda_{\sg_i} \cdot \prod \limits _{k=0} ^m\lambda_{\sg_k}\inv\Big) \\
  &= \delta \inv\Big(\prod \limits _{i=0} ^n \lambda_{\sg_i} \cdot \prod \limits _{j=0} ^m \lambda_{\sg_j}\inv \cdot \prod \limits _{j=0} ^m\lambda_{\sg_j} \cdot \prod \limits _{k=0} ^l\lambda_{\sg_k}\inv\Big) \\
  &= \delta \inv\Big(\prod \limits _{i=0} ^n \lambda_{\sg_i} \cdot \prod \limits _{j=0} ^m \lambda_{\sg_j}\inv\Big) \circ \delta \inv \Big( \cdot \prod \limits _{j=0} ^m\lambda_{\sg_j} \cdot \prod \limits _{k=0} ^l\lambda_{\sg_k}\inv\Big) \\
  &= \mT(M) \circ \mT(N)
\end{split} \eeq and (2) is satisfied. A connected cylinder cobordism of $\sg$ is a cobordism $\sg \rsa \sg$, so we have \beq \begin{split} \mT(\sg \x I) & =\delta \inv(\lambda_{\sg} \cdot \lambda_{\sg} \inv) \\
  & =\delta \inv(1) \\
  & = id_{\bk^\x} \end{split} \eeq so (3) is satisfied. The disconnected case follows from (4), which clearly follows from Equation \eqref{12}. (5) is satisfied trivially.
  
Thus $\mT$ is an invertible $n$-TQFT, and for all closed manifolds $M$, \beq \begin{split} \mT(M) & =\delta\inv(\lambda_\emp \circ \lambda_\emp \inv) \\
  & =\delta \inv(1) \\
  & =id_{\bk ^\x}. \end{split} \eeq We chose $\mT$ to take on nontrivial values, so $\mT$ is not the trivial TQFT.
\epf 
We now want to try to figure out the image of $\Psi_n$. One problem, of course, comes from the ``forgetfullness" of $\Psi_n$. Given an $SKK$ invariant $\xi$, we would like to choose an invertible $n$-TQFT $\mT$ such that $\Psi_n(\mT)= \xi$. But $\xi$ gives us no explicit information as to how $\mT$ should evaluate cobordisms with nonempty boundary.

Some of the $SKK$ invariants that we've listed, however, do give us information about how to evaluate such cobordisms. We cannot with perfect accuracy say that any of these invariants define TQFTs, since these invariants (composed with exponential functions, if necesarry) assign nonzero scalars to manifolds rather than invertible linear maps $\bk \to \bk$. This difference is, however, superficial, and we therefore choose to ignore it. Specifically, if $\label{003}\mM^{\pl}_n$ denotes the set of diffeomorphism classes of compact oriented $n$-manifolds with boundary, then we will say that a function  $\Theta : \mM^{\pl}_n \to \bk^\x$ defines an invertible $n$-TQFT if the formula \beq \mT(M)=\delta \inv (\Theta(M)) \eeq defines an invertible $n$-TQFT. The following proposition shows us exactly when $\Theta$ has this property.

\begin{theorem} \label{first} The function $\Theta : \mM^{\pl}_n \to \bk^\x$ defines an invertible $n$-TQFT if and only if \beq  \Theta([M \underset{f}{\cup}N])=\label{above} \Theta([M]) \cdot \Theta([N]) \eeq for all $[M], [N] \in \mM^{\pl}_n$. Here $f$ is any orientation-preserving diffeomorphism $\pl_{\textnormal{out}} M \to \overline{\pl_{\textnormal{in}} N}$, where $\pl_{\textnormal{out}}M$ and $\pl_{\textnormal{in}}N$ are unions of  boundary components of $M$ and $N$, respectively.
\end{theorem}
\begin{proof} 

First suppose that Equation \eqref{above} holds. We check the TQFT axioms of Definition \ref{dTQFT}.

(1) is clearly satisfied, since for any equivalent $n$-cobordisms $M$ and $N$, there is an orientation-preserving diffeomorphism $\psi: M \to N$. $\Theta$ must evaluate such cobordisms equally. (2) is satisfied by Equation \eqref{above}. To show (3), note that \beq \Theta([\sg \times I])= \Theta([\sg \times I]) \cdot \Theta([\sg \times I]) \eeq for all closed $(n-1)$-manifolds $\sg$. This true because the two identical cylinder cobordisms can be glued to produce another cylinder cobordism. Thus we have $\Theta([\sg \times I])=1$. (4) is also satisfied by Equation \eqref{above}, where $f$ is an empty map. (5) is satisfied trivially.

Now suppose $\Theta$ defines an invertible $n$-TQFT $\mT$. Then $\mT$ must evaluate $n$-cobordisms based only on their oriented diffeomorphism class. Now let $f:\pl_{\textnormal{out}}M \to \overline{\pl_{\textnormal{in}}N}$ be an orientation-preserving diffeomorphism. We can easily form cobordisms $M: \overline{\pl_{\textnormal{in}}M} \rsa \pl_{\textnormal{out}}M$ and $N: \pl_{\textnormal{out}}M \rsa \pl_{\textnormal{out}}N$, where the in-boundary diffeomorphism of $N$ is $f$ and all other diffeomorphisms are the identity. The resulting glued cobordism has \beq M \underset{f}{\cup}N \eeq as its $n$-manifold. Thus by axiom (2) of Definition \ref{dTQFT}, Equation \eqref{above}  holds. \end{proof}

\begin{corollary}\label{second} If an oriented $n$-diffeomorphism invariant $\Theta$ on manifolds with boundary defines an invertible $n$-TQFT, then it restricts to a linear combination of the Euler characteristic and signature on closed manifolds.
\end{corollary}
\bpf Because our choice of $f$ was arbitrary in Theorem \ref{first}, $\Theta$ restricts to an $SK$ invariant on closed manifolds. The result follows from Theorem \ref{phew}.
\epf

\begin{corollary} The Euler characteristic $\chi$ and semicharacteristic $\chi_{1/2}$ define invertible $n$-TQFTs if and only if $n$ is even. The signature $\sigma$ defines an $n$-TQFT. 
\end{corollary}
\bpf For $n$ even, $\chi$ and $\chi_{1/2}$ satisfy Equation \eqref{above} by the union formula for the Euler characteristic and the fact that the Euler characteristic of any closed odd-dimensional manifold is zero. The result then follows from Theorem \ref{first}. 

Now let $n$ be odd. The $n$-disk has Euler characteristic 1. We can glue two $n$-disks via the identity of the boundary to form an $n$-sphere, which has Euler characteristic 0, contradicting Equation \eqref{above}. Thus the Euler characteristic cannot does not define an invertible $n$-TQFT for $n$ odd.

That  $\chi_{1/2}$ cannot does not define an $n$-TQFT for $n$ odd follows from Corollary \ref{second}. 

The signature $\sigma$ satisfies Equation \eqref{above} by Novikov additivity. 
\epf

Note that the above corollary is not a proof that $\Psi_n$ is not surjective. We have only given conditions for an \tit{oriented diffeomorphism invariant} on manifolds with boundary to define a TQFT. In general, an $n$-TQFT can pick up more information than just the oriented diffeomorphism class of the $n$-manifold; in particular, it notices the choice of boundary manifolds and diffeomorphisms. These choices completely determine how cobordisms are glued. This is why $\Theta$ does not notice how manifolds are glued.

\medskip

We would like to describe\beq \Psi_n:\p{nTQFT} ^\x _\bk \to \Hom{(SKK_n, \bk^\x)}\eeq by giving an  exact sequence that includes these terms and the kernel and cokernel of $\Psi_n$. This is difficult to do in general, but under a few assumptions it is feasible. In particular, we shall see that if the target group of the $SKK$ invariants is the multiplicative group of positive reals, then we can describe the sequence. Before we go into the description of this exact sequence we give a preliminary definition and lemma.

\begin{definition} \label{07} Let $M$ be a compact oriented manifold with boundary. We define the \tit{double} of $M$ to be the closed manifold  \beq D(M)= M \underset{id_{\pl M}}{\cup} \overline{M}. \eeq 
\end{definition}


The following lemma gives a relation in $SKK_n$ that will be necessary for the description of our exact sequence.

\begin{lemma} \label{LEM} Let $X_1$, $X_2$, and $X_3$ be oriented $n$-manifolds with boundaries $\sg_1$, $\sg_2$ and $\sg_3$, respectively. An example is pictured in Figure \ref{22}.

\begin{figure} [h]
   \centering
   \def\svgwidth{\columnwidth}
   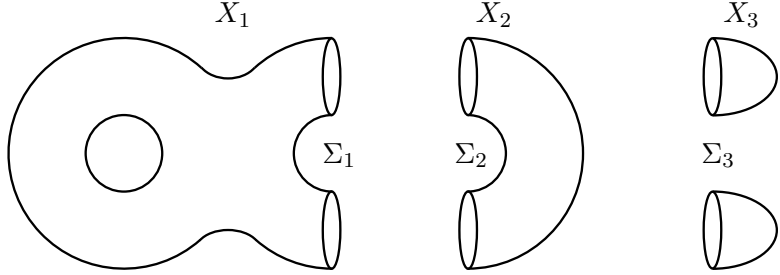
   \caption{$X_1$, $X_2$, and $X_3$}
   \label{22}
\end{figure} 

 Let \beq f: \sg_1 \to {\sg}_2 \ \mathrm{and} \ g: \sg_2 \to \overline{\sg}_3 \eeq be orientation-preserving diffeomorphisms. Then in $SKK_n$, we have \beq [X_1 \underset{f}{\cup} \overline{X}_2] + [{X}_2 \underset{g}{\cup} {X_3}]=[X_1 \underset{g \circ f}{\cup} {X_3}] + [D(X_2)]. \eeq The desired relation is pictured in Figure \ref{24}.

 \begin{figure} [h]
   \centering
   \def\svgwidth{\columnwidth}
   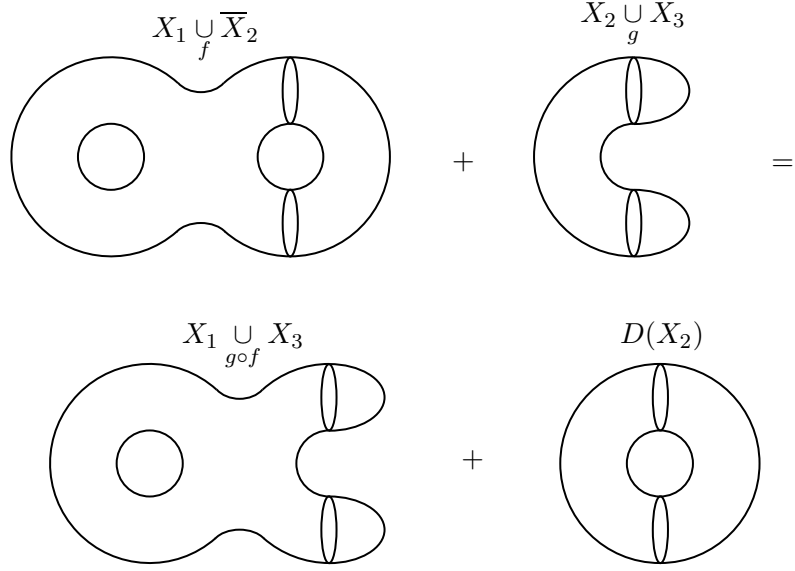
   \caption{Our desired relation in $SKK_n$}
   \label{24}
\end{figure}
 \end{lemma}

\bpf Recall that  in $SKK_n$, whenever $M$ and $N$ are pasted, respectively, via the same diffeomorphisms as $M'$ and $N'$, \beq [M] -[N] = [M'] -[N']. \label{16} \eeq We will make use of this equation to prove the lemma. To do this, we must start with two ``cut" manifolds with boundary, glue each in two different ways, and apply \eqref{16}. 
	Our starting manifolds will be \beq X^*_1= X_1 \coprod X_2 \coprod {X}_2' \coprod \overline{X}_3, \eeq which is pictured in Figure \ref{25},
	
	\begin{figure} [h!]
   \centering
   \def\svgwidth{\columnwidth}
   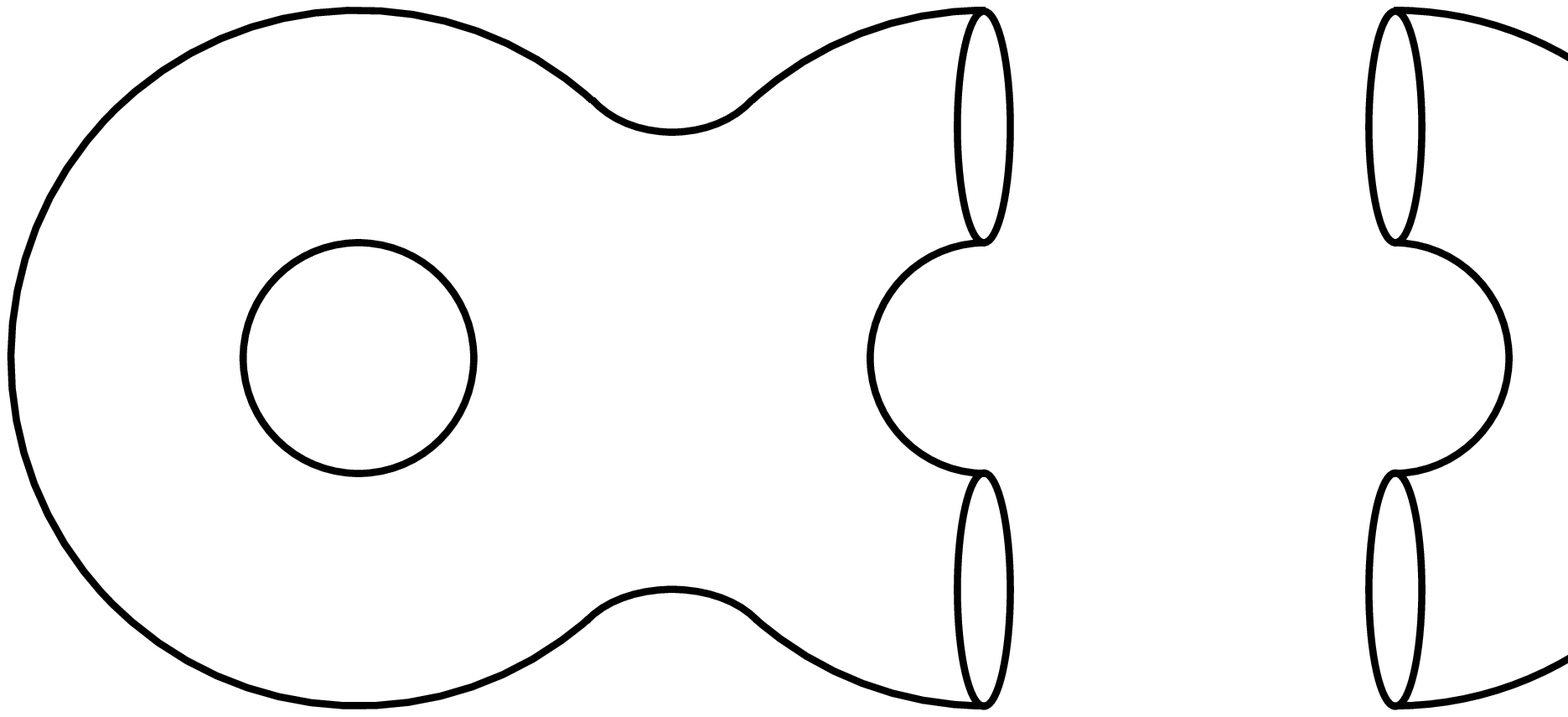
   \caption{$X^*_1$}
   \label{25}
\end{figure}

	\noindent and \beq X^*_2=X_1 \coprod X_2 \coprod {X}_2'\coprod C_{g\inv} \underset{id_{\sg_2}}{\cup}\overline{X}_2, 	\eeq where $C_{g \inv}$ is the mapping cylinder of $g \inv$. This manifold is pictured in Figure \ref{26}. 
	 
	 \begin{figure} [h!]
   \centering
   \def\svgwidth{\columnwidth}
   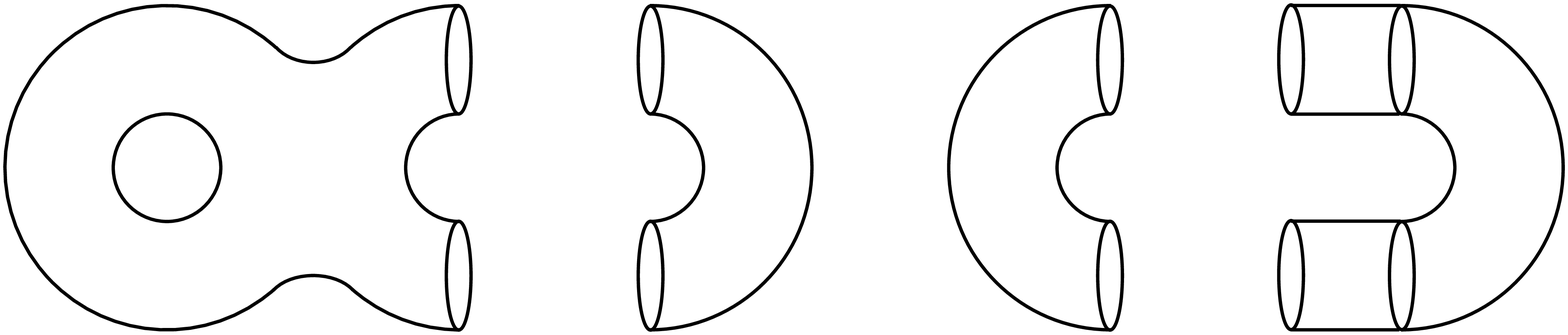
   \caption{$X^*_2$}
   \label{26}
\end{figure}

    	 Both have \beq \sg_1 \coprod \sg_2 \coprod \sg_2' \coprod \overline{\sg}_3 \eeq as their boundaries. The two gluing diffeomorphisms are \beq 	F_1 = \begin{cases} f: &\sg_1 \to \sg_2 \\ g: &  {\sg}_2'\to \overline{\sg}_3 \end{cases}  \ \ \mathrm{and} \ \   F_2 = \begin{cases} g \circ f : & \sg_1 \to \overline{\sg}_3 \\ id:& \sg_2 \to \sg_2' \end{cases} \eeq Gluing 	$X^*_1$ via $F_1$, we obtain \beq [X_1 \underset{f}{\cup} \overline{X}_2] + [X_2 \underset{g}{\cup} {X_3}], \eeq  as in Figure \ref{29}.
	 
	 \begin{figure} [h!]
   \centering
   \def\svgwidth{\columnwidth}
   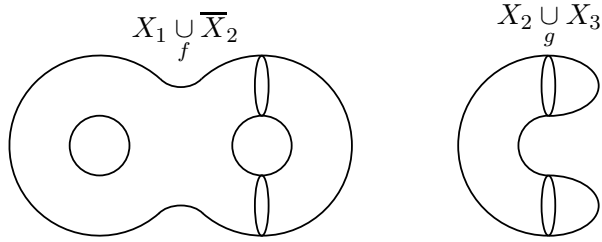
   \caption{$X^*_1$ glued via $F_1$}
   \label{29}
\end{figure} 
	 
	\noindent Gluing $X^*_1$ via $F_2$, we obtain \beq [X_1 \underset{g \circ f}{\cup} {X_3}] + [D(X_2)], \eeq as in Figure \ref{37}.
	
	\begin{figure} [h!]
   \centering
   \def\svgwidth{\columnwidth}
   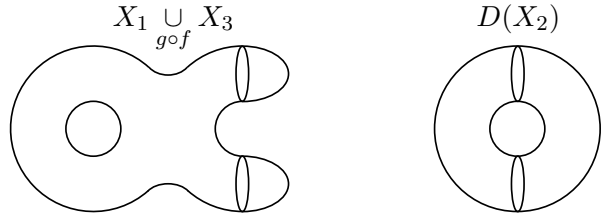
   \caption{$X^*_1$ glued via $F_2$}
   \label{37}
\end{figure} 
	
	\noindent Gluing $X^*_2$ via $F_1$, we obtain \beq [X_1 \underset{f}{\cup} \overline{X}_2] +[D(X_2)], \eeq as in Figure \ref{27}.
	
\begin{figure} [h!]
   \centering
   \def\svgwidth{\columnwidth}
   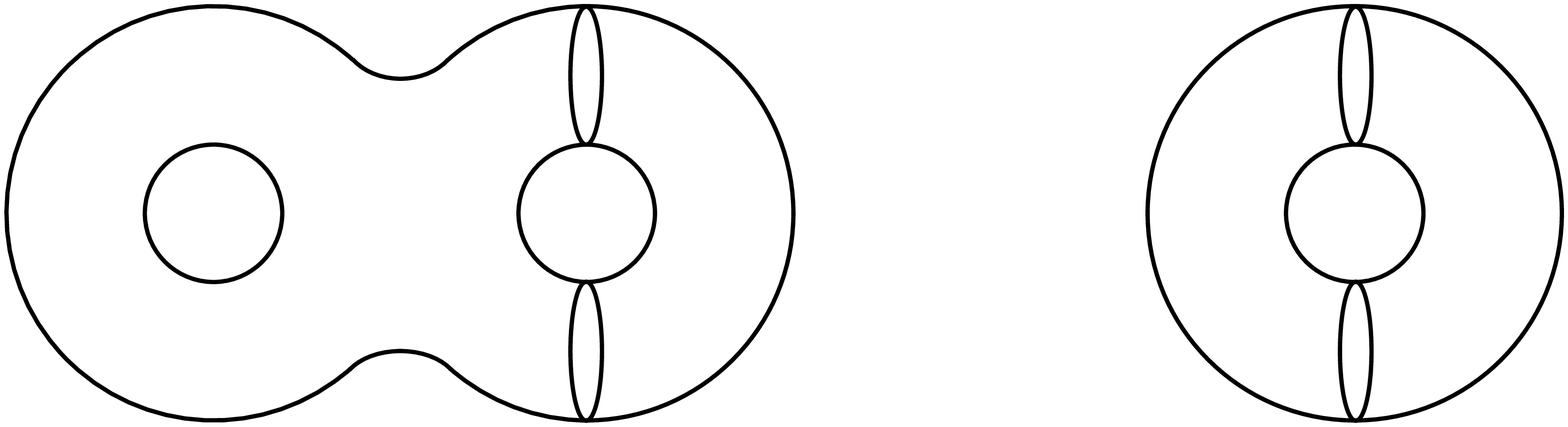
   \caption{$X^*_2$ glued via $F_1$}
   \label{27}
\end{figure} 
	
	\noindent Gluing $X^*_2$ via $F_2$, we obtain \beq [X_1 \underset{f}{\cup} \overline{X}_2] +[D(X_2)], \eeq as in Figure \ref{28}. \begin{figure} [h!]
   \centering
   \def\svgwidth{\columnwidth}
   \input{lemma16.eps_tex}
   \caption{$X^*_2$ glued via $F_2$}
   \label{28}

 \vspace*{1mm}

\end{figure} 

\noindent Now we apply Equation \eqref{16}. We have \beq  \begin{split} [X_1 \underset{f}{\cup} \overline{X}_2] + [X_2 \underset{g}{\cup} {X_3}]-[X_1 \underset{g \circ f}{\cup} {X_3}] -[D(X_2)]& = [X_1 \underset{f}{\cup} \overline{X}_2] +[D(X_2)] \\ 
& \quad -[X_1 \underset{f}{\cup} \overline{X}_2] -[D(X_2)]. \end{split} \nonumber \eeq
Therefore
\beq \begin{split}
 [X_1 \underset{f}{\cup} \overline{X}_2] + [X_2 \underset{g}{\cup} {X_3}]-[X_1 \underset{g \circ f}{\cup} {X_3}] -[D(X_2)]& = 0,  \end{split} \nonumber \eeq
and
\beq \begin{split}
[X_1 \underset{f}{\cup} \overline{X}_2] + [X_2 \underset{g}{\cup} {X_3}]& =[X_1 \underset{g \circ f}{\cup} {X_3}] + [D(X_2)], \end{split} \nonumber \eeq as desired. \end{proof}

We now have everything we need to describe our sequence.

\begin{theorem} Let $\mathscr{T}^n_\bR(1,-1)$ denote the subgroup of  $\p{nTQFT}^\x_{\bR}$ that consists of all invertible $n$-TQFTs that take on the values of $\delta \inv (1)$ and $\delta \inv (-1)$ on closed manifolds. Then the following sequence is split exact: \beq 0\to \mathscr{T}^n_\bR(1,-1) \xrightarrow{i_n} \p{nTQFT}^\x_{\bR}  \xrightarrow{\mid \Psi_n \mid} \Hom{(SKK_n, \bR^\x_+)} \to 0 \eeq Here $i_n$ is the inclusion map, $\bR_+^\x$ is the multiplicative group of positive reals, and $|\Psi_n|$ is the homomorphism \beq | \Psi_n |(\mT) := [M]_{SKK} \to | \delta(\mT([M]_{SKK}))|. \eeq
\end{theorem}
\begin{proof} Checking exactness at the term $\p{nTQFT}^\x_{\bR}$ is easy. Because of the absolute value bars in the definition of $| \Psi_n |$, it is clear that $| \Psi_n| \circ i_n$ is trivial. Also, if $|\Psi_n|(\mT)$ is the trivial $SKK_n$ invariant, then \beq | \delta(\mT([M]_{SKK}))|=1 \eeq on all closed manifolds $X$. This shows that  \beq \mT([M]_{SKK})=\delta \inv ( \pm 1) \eeq and $\mT \in \mathscr{T}^n_\bR(1,-1)$.

The difficult part is defining the splitting homomorphism \beq \mS: \Hom{(SKK_n, \bR_+^\x)} \to \p{nTQFT} ^\x _\bR. \eeq That is, given an $n$-dimensional $SKK$ invariant $\xi$ with values in the positive reals, we must find an invertible TQFT $\mS(\xi)$ such that $|\Psi_n|(\mS(\xi))=\xi$. This cannot be a completely arbitrary assignment, because $\mS$ must be a homomorphism.

We first make use of Theorem \ref{class} to describe $\Hom{(SKK_n, \bR_+^\x)}$. We let $\chi^*$ denote the subgroup of $\Hom{(SKK_n, \bR_+^\x)}$ given by multiples of the exponential map composed with the Euler characteristic. Since $\bR_+^\x$ is torsion-free, we have \beq \Hom{(SKK_n, \bR_+^\x)} =\begin{cases} \qquad \quad \  0 & n \ \mathrm{odd} \\ \qquad \quad \hspace{.3mm} \chi^* & n \equiv 2 \ \mathrm{mod}  \ 4 \\ \chi^* \oplus \Hom{(\Omega_n, \bR_+^\x)} & n \equiv 0 \ \mathrm{mod} \ 4 \end{cases} \eeq   Here we are using the fact that $\Om_n$ is finite for $n \not\equiv 0 \mod 4$.

We know exactly how to define $\mS$ on the direct summands given by $\chi^*$; $\exp \circ \chi$ is a generator of $\chi^*$, so we simply take $S(\chi)(M)= \delta \inv(\exp{(\chi(M)))}$ for a cobordism with $n$-manifold $M$. By Theorem \ref{first}, this defines an invertible TQFT, and we obviously have \beq |\Psi_n| \circ \mS = id \eeq on these direct summands. So it only remains to define $\mS$ on  $\Hom{(\Omega_{n}, \bR_+^\x)}$ where $4|n$. That is, we can assume $4|n$ and $\xi \in \Hom{(\Omega_{n}, \bR_+^\x)}$. Now we take the necessary steps to define $\mT$.

Under these assumptions, $\Omega_{n-1}$ is finite. Let $l$ be the order of $\Omega_{n-1}$. Then for every closed connected $(n-1)$-manifold $\sg$, we can choose an $n$-manifold $B_\sg$ with boundary given by \beq \pl B_\sg= \coprod_l \sg. \eeq For simplicity we require $B_\emp=\emp$.  Now let \beq M:\coprod \limits _{i=0} ^n \sg_i \rsa \coprod \limits _{j=0} ^m \sg_j \eeq be a cobordism, where each $\sg_i$ and $\sg_j$ is connected. Let $\phi_{i,{\textnormal{in}}}$ and $\phi_{j,{\textnormal{out}}}$ be the in-boundary and out-boundary diffeomorphisms of the components of $\pl M$. Before we give the definition, we form a closed manifold $C(M)$ as follows: Let $\phi^*_{\textnormal{in}}$ be the orientation-preserving diffeomorphism \beq \coprod_{i=0} ^n \pl B_{\sg_i} \to \coprod_l \pl_{\textnormal{in}}\overline{M} \eeq given by the appropriate $\phi_{i, {\textnormal{in}}}$ on each component. Also let $\phi^*_{\textnormal{out}}$ be the orientation-preserving diffeomorphism \beq \coprod_{j=0} ^m \pl {B}_{\sg_j} \to \coprod_l \pl_{\textnormal{out}}M \eeq given by the appropriate $\phi_{j,{\textnormal{out}}}$ on each component. Now define \beq C(M)= \Big( \coprod \limits _{i=0} ^n B_{\sg_i} \Big) \  \underset{\phi_{\textnormal{in}}^*}{\bigcup} \ \Big( \coprod_l M \Big) \  \underset{(\phi_{\textnormal{out}}^*) \inv}{\bigcup} \ \Big( \coprod \limits _{j=0} ^m \overline{B}_{\sg_j} \Big). \eeq Now we can define $\mS(\xi)(M)$ by the following equation \beq \label{def} \mS(\xi)(M)=\delta \inv( \xi(C(M))^{1/l}) \eeq

$\mS$ is clearly a homomorphism. We need to check the axioms of Definition \ref{dTQFT} to show that $\mS(\xi)$ is an $n$-TQFT. Checking (1) first, suppose that \beq M: \coprod \limits _{i=0} ^n \sg_i \rsa \coprod \limits _{j=0} ^m \sg_j \ \mathrm{and} \ N: \coprod \limits _{i=0} ^n \sg_i \rsa \coprod \limits _{j=0} ^m \sg_j \eeq are equivalent coborisms. We will set \beq \sg_0= \coprod \limits _{i=0} ^n \sg_i \ \mathrm{and} \ \sg_1= \coprod \limits _{j=0} ^m \sg_j \eeq for notational simplicity. In addition, we will have $\phi_{\textnormal{in}}$ and $\phi_{\textnormal{out}}$ be the boundary diffeomorphisms of $M$ and $\phi_{\textnormal{in}}'$ and $\phi_{\textnormal{out}}'$ be the boundary diffeomorphisms of $N$. Since $\xi$ evaluates diffeomorphic manifolds equally, we need only show that $C(M)$ and $C(N)$ are diffeomorphic. Since the in- and out-boundaries of $M$ and $N$ are the same, (1) will hold by Equation \eqref{def}. We have diffeomorphisms for each of the gluing components of $C(M)$, which include \beq \begin{aligned} 
id: & \coprod \limits _{i=0} ^n B_{\sg_i}  &\to& \ \  \coprod \limits _{i=0} ^n B_{\sg_i} \\
\psi: & \ N  &\to& \ \  M \\
id: & \coprod \limits _{j=0} ^m B_{\sg_j}  &\to&\ \  \coprod \limits _{j=0} ^m B_{\sg_j}\end{aligned} \label{diffeo} \eeq for each of the gluing components of $C(M)$. Recall that the diagram
\bcomd & M & \\
\sg_0 \arrow{ur}{\phi_{\textnormal{in}}} \arrow{dr}[swap]{\phi'_{\textnormal{in}}} & & \sg_1 \arrow{dl}{\phi'_{\textnormal{out}}} \arrow{ul}[swap]{\phi_{\textnormal{out}}} \\
& N \arrow{uu}[swap]{\psi}& \label{40}
\ecomd commutes. This amounts to the assertion that these diffeomorphisms of \ref{diffeo} agree on the glued areas of $C(M)$ and $C(N)$. Thus $C(M)$ and $C(N)$ are diffeomorphic, and (1) holds. 

Now we check (2). Let \beq M:  \coprod \limits _{i=0} ^n \sg_i \rsa \coprod \limits _{j=0} ^m \sg_j \  \mathrm{and} \ N:  \coprod \limits _{j=0} ^m \sg_j \rsa \coprod \limits _{k=0} ^p \sg_k \eeq be two cobordisms. To prove (2), we apply Lemma \ref{LEM}. Keeping the notation from the construction of $C(M)$, set \beq 
\begin{split} X_1 & = \Big(\coprod _{i=0} ^n B_{\sg_i}\Big) \underset{\phi_{\textnormal{in}}^*}{\bigcup} \Big(\coprod _l M \Big) \\
X_2 & =\coprod _{j=0} ^m B_{\sg_j} \\
X_3 &=  \Big( \coprod _l N \Big) \underset{(\phi_{\textnormal{out}}^{*'}) \inv}{\bigcup}\Big(\coprod _{i=0} ^n \overline{B}_{\sg_i}\Big) \\
f &= (\phi_{\textnormal{out}}^*) \inv \\
g &= \phi_{\textnormal{in}}^{*'}. \end{split} \eeq

\noindent Lemma \ref{LEM} gives \beq [X_1 \underset{f}{\cup} \overline{X}_2] + [X_2 \underset{g}{\cup} {X_3}] =[X_1 \underset{g \circ f}{\cup} {X_3}] + [D(X_2)] \eeq in $SKK_n$, which gives in our case \beq [C(M)] + [C(N)]= [C(MN)] + [D(X_2)]. \eeq

\noindent Now note that because $[D(X_2)]\label{008}$ bounds and $\xi \in \Hom{(\Omega_n, \bR_+^\x)}$, we have $\xi([D(X_2)])=1$. Thus \beq \begin{split} 
\mS(\xi)(MN) &= \delta \inv( \xi(C(MN))^{1/l}) \\
 & =\delta \inv(\xi(C(M))^{1/l} \cdot \xi(C(N))^{1/l}) \\
 & =\delta \inv( \xi(C(M))^{1/l}) \circ \delta \inv( \xi(C(N))^{1/l}) \\
 &=\mS(\xi)(M) \circ \mS(\xi)(N), \end{split} \eeq as desired. Thus (2) holds.
 
 It should be clear that \beq C(M \coprod N)=C(M) \coprod C(N). \eeq From this (3) clearly follows. For (4), note that for a cylinder cobordism \beq \Big( \coprod \limits _{i=0} ^n \sg_i\Big) \x I:  \coprod \limits _{i=0} ^n \sg_i \rsa  \coprod \limits _{i=0} ^n \sg_i, \eeq we have \beq C\Big{(\Big(} \coprod \limits _{i=0} ^n \sg_i\Big) \x I \Big)= D \Big( \coprod _{i=0} ^n B_{\sg_i} \Big). \eeq Thus \beq \begin{split} \mS(\xi)\Big(\Big( \coprod \limits _{i=0} ^n \sg_i\Big) \x I \Big)& = \delta \inv\Big( \xi\Big(C\Big(\Big( \coprod \limits _{i=0} ^n \sg_i\Big) \x I \Big)\Big)^{1/l}\Big) \\
  & = \delta \inv\Big( \xi\Big(D\Big( \coprod _{i=0} ^n B_{\sg_i} \Big)\Big)^{1/l}\Big)  \\
  & = \delta \inv(1) \\
  & = id_{\bR}, \end{split} \eeq and (4) holds. (5) holds trivially. Thus $\mS(\xi)$ is an invertible $n$-TQFT. 
  
  Lastly, we need to check that \beq |\Psi_n| \circ \mS(\xi)=\xi \eeq for all $\xi \in \Hom{(\Omega_n, \bR_+^\x)}$. By our requirement that $B_\emp=\emp$, we have for each closed manifold $M$ \beq C(M)=\coprod_l M. \eeq Thus \beq \begin{split} |\Psi_n| \circ \mS(\xi)(M) &= |\Psi_n|( \delta \inv( \xi(C(M))^{1/l})) \\
   &= |\Psi_n|( \delta \inv(\xi(M))) \\
   &= \xi(M). \end{split} \eeq \end{proof}

The construction of the splitting homomorphism $\mS$ is \tit{not} independent of the choices of the $B_\sg$'s. To illustrate this, let $\xi \in \Hom{(SKK_8, \bR^\x_+)}$ be the bordism invariant given by $\xi(M)=\exp{(p_2(M))}$, where $p_2$ denotes the Pontryagin number given by the trivial partition of a two-element set. Also let $D^8$ and $\mathring{D^8}$ denote the closed and open $8$-disks respectively. Now consider the cobordism \beq {D^8}: \emp \rsa S^7. \eeq  Since $\Om_7$ has order $1$,  $D^8$ and $ \bC P^4-\mathring{D^8}$ are both sufficient choices for $B_{S^7}$. Now defining $\mS$ with the choice $D^8$, we have \beq \begin{split} \mS(\xi)(D^8) &= \delta \inv (\xi(S^8)) \\ & = \delta \inv (\exp{(p_2(S^8))}) \\ & = \delta \inv (\exp{(0)}) \\ & =id_{\bR}. \end{split} \eeq Defining $\mS$ with the choice $ \bC P^4-\mathring{D^8}$, we have \beq \begin{split} \mS(\xi)(D^8) &= \delta \inv (\xi(\bC P^4)) \\ & = \delta \inv (\exp{(p_2(\bC P^4))}) \\ & = \delta \inv (\exp{(10)}) \\ &\neq id_{\bR}. \end{split} \eeq Given a closed manifold $\sg$ that bounds, there is no canonical choice of manifold $M$ with $\pl M = \sg$. This is why we do not explicitly define $\mS$. An explicit definition is not necessary, however, for showing that the sequence splits.

%
%
%
%
%

\bibliographystyle{alpha}

\bibliography{biblio-REU}

\end{document}